\documentclass[11pt]{amsart}
\usepackage{amssymb,latexsym,amsmath,amsfonts}
\usepackage[mathscr]{eucal}
\numberwithin{equation}{section}
\theoremstyle{definition}
\newtheorem{definition}{Definition}[section]
\newtheorem{example}[definition]{Example}

\theoremstyle{remark}
\newtheorem{remark}[definition]{Remark}
\theoremstyle{plain}
\newtheorem{theorem}[definition]{Theorem}
\newtheorem{lemma}[definition]{Lemma}
\newtheorem{proposition}[definition]{Proposition}

\newcommand{\eps}{\varepsilon}

\newcommand{\zahl}{\mathbb{Z}}

\newcommand{\bdy}{\partial}

\newcommand\ball[1]{\mathbb{B}^{#1}}

\newcommand{\OM}{\Omega}
\newcommand{\Dee}{\mathbb{D}}
\newcommand{\polyD}{\triangle}

\newcommand{\apprh}{\mathscr{A}_{\alpha,N}}
\newcommand{\loc}{\OM_F\cap\triangle}

\newcommand{\smoo}{\mathcal{C}}
\newcommand{\hol}{\mathcal{O}}

\newcommand{\elII}{\mathbb{L}^2}

\newcommand{\bcdot}{\boldsymbol{\cdot}}
\newcommand{\er}{\mathfrak{Re}}
\newcommand{\mi}{\mathfrak{Im}}

\newcommand{\mapp}{\longrightarrow}

\newcommand{\lrarw}{\longrightarrow}
\newcommand{\geef}{G_f}

\newcommand{\lamf}{\Lambda_f}
\newcommand{\poet}{{\kappa_\eta}}
\newcommand\lint[2]{\int\limits_{{#1}}^{{#2}}}

\newcommand{\CC}{\mathbb{C}^2}
\newcommand{\Cn}{\mathbb{C}^n}
\newcommand{\cplx}{\mathbb{C}} 
\newcommand{\RR}{\mathbb{R}}

\begin{document}

\title[Growth of the Bergman kernel]{On the growth of the Bergman kernel \\
near an infinite-type point}
\author{Gautam Bharali}
\address{Department of Mathematics, Indian Institute of Science, Bangalore -- 560 012}
\email{bharali@math.iisc.ernet.in}
\thanks{This work is supported in part by a grant from the UGC under DSA-SAP, Phase~IV.}

\begin{abstract} 
We study diagonal estimates for the Bergman kernels of certain model domains in $\CC$ near
boundary points that are of infinite type. To do so, we need a mild structural
condition on the defining functions of interest that facilitates {\em optimal}
upper and lower bounds. This is a mild condition; unlike earlier studies of this sort, we
are able to make estimates for {\em non-convex} pseudoconvex domains as well. This condition
quantifies, in some sense, how flat a domain is at an infinite-type boundary point. In this
scheme of quantification, the model domains considered below range --- roughly speaking --- from
being ``mildly infinite-type'' to very flat at the infinite-type points.
\end{abstract}
\maketitle

\section{Statement of Results}\label{S:intro}

Let $\OM\subset\CC$ be a pseudoconvex domain (not necessarily bounded) having a smooth boundary.
Let $p\in\bdy\OM$ be a point of infinite type: by this we mean that for {\em each} $N\in\zahl_+$,
there exists a germ of a $1$-dimensional complex-analytic variety through $p$ whose order of 
contact with $\bdy\OM$ at $p$ is at least $N$. If $\bdy\OM$ is not Levi-flat around $p$, 
there exist holomorphic coordinates $(z,w;V_p)$ centered at $p$ such that
\begin{equation}\label{E:local}
\OM\bigcap V_p \ = \ \{(z,w)\in V_p:\mi w > F(z)+R(z,\er w)\},
\end{equation}
where $F$ is a smooth, subharmonic, non-harmonic function defined in a neighbourhood of $z=0$,
that vanishes to infinite order at $z=0$; $R(\bcdot \ ,0)$ vanishes to 
infinite order at $z=0$; and 
$R$ is $O(|z||\er w|,|\er w|^2)$. Given the infinite order of vanishing of $F$ at $z=0$, how
does one find estimates for the Bergman kernel of $\OM$ near $p$~? In many cases --- for instance:
when $\bdy\OM\cap V_p$ is pseudoconvex of strict type, in the sense of 
\cite{kohn:bb/dbar/wpmd2:72}, away from $p\in\bdy\OM$ --- the function $F$ in \eqref{E:local}
can be extended to a global subharmonic function. In such situations, the model domain
\begin{equation}\label{E:mod}   
\OM_F \ := \ \{(z,w)\in\CC:\mi w > F(z)\},
\end{equation}
approximates $\bdy\OM$ to infinite order along the complex-tangential directions at $p$.
One is thus motivated to investigate estimates for the Bergman kernel for domains of the
form \eqref{E:mod}. In this paper, we shall find estimates for the Bergman kernel of $\OM_F$
on the diagonal as one approaches $(0,0)\in\bdy\OM_F$. More specifically:
\begin{itemize}
\item[\textbullet] We shall derive estimates 
that hold {\em not just} in a non-tangential interior cone
with vertex at $(0,0)$, but for a family of much larger approach regions that comprises regions with
{\em arbitrarily high} orders of contact at $(0,0)$; and
\item[\textbullet] We shall find optimal estimates for the 
growth of the kernel (evaluated on the diagonal
of $\OM_F\times\OM_F$) as $(z,w)\lrarw(0,0)$ through any of the aforementioned approach regions.
\end{itemize}
\smallskip
Pointwise estimates, and a lot more, have been obtained for finite-type domains in $\CC$;
see for instance \cite{diedHerOhs:Bkuepd86} by Diederich {\em et al};
\cite{nagelRosaySteinWainger:eBkSkcwpd88} and \cite{nagelRosaySteinWainger:eBSkC289}
by Nagel {\em et al}; and \cite{mcneal:bbBkfC289} by McNeal. In \cite{kimLee:abBkaicitpd02},
Kim and Lee provide some estimates on the diagonal for the Bergman kernel, as one approaches
an infinite-type boundary point, for a class of convex, infinite-type domains in $\CC$. However,
to the best of our knowledge, not even pointwise estimates are known for any reasonably general
class of pseudoconvex (not necessarily convex) domains of infinite type. Determining such
estimates even for model domains of the form \eqref{E:mod} is not so easy. For instance,
techniques analogous to the scaling methods used in the papers \cite{mcneal:bbBkfC289} and
\cite{nagelRosaySteinWainger:eBSkC289}, in \cite{boasStraubeYu:blBkm95} by Boas {\em et al},
and in \cite{krantzYu:BicBm96} by Krantz-Yu do not seem to yield optimal estimates. Another
problem is that we do not know {\em a priori} whether $\OM_F$ --- recall that our models
{\em do not} arise as limits of scalings of bounded domains --- even has a non-trivial
Bergman space. Things become tractable if we impose a simplifying condition on $F$:
\[
(*) \ \begin{cases}
        \ \text{$F$ is a radial function, i.e. $F(z)=F(|z|) \ \forall z\in\cplx$, and}\\
        \ \text{$\exists\eta>0$ such that $F(z)\geq C|z|^\eta$ when $|z|\geq R$ for some $R>0$ and
                $C>0$.}
        \end{cases}
\]
Under this condition, $\OM_F$ has a non-trivial Bergman space; see for instance 
\cite{haslinger:BHsmd} by Haslinger.
However, given the condition $(*)$, we can say more: under this condition, $\OM_F$ has
a bounded realization and thus admits a localization principle for the Bergman kernel.
To state this precisely, we recall that the Bergman projection for $\OM$ is the orthogonal
projection $B_\OM:\elII(\OM)\mapp\hol(\OM)\cap\elII(\OM)$, and the Bergman kernel is the kernel
representing this projection. Let us denote the Bergman kernel of 
$\OM_F$ as $B_F(Z,Z^\prime), \ (Z,Z^\prime)\in\OM_F\times\OM_F$. We will denote the kernel 
restricted to the diagonal by $K_F(z,w):=B_F((z,w),(z,w))$. We can now state
\smallskip

\begin{proposition}\label{P:locali} 
Let $F$ be a $\smoo^\infty$-smooth subharmonic function that vanishes to infinite order at 
$0\in\cplx$ and satisfies the condition $(*)$. Assume that the boundary of the domain
$\OM_F:=\{(z,w)\in\CC:\mi w > F(z)\}$ is not Levi-flat around the origin. Then:
\begin{enumerate}
\item[1)] There exists an injective holomorphic map $\Psi$ defined in a neighbourhood
of $\overline{\OM}_F$ such that $\Psi(\OM_F)$ is a bounded pseudoconvex domain.
\item[2)] For each polydisc $\polyD$ centered at the origin, there exists a 
constant $\delta\equiv\delta(\polyD)>0$ such that
\begin{equation}\label{E:locali}
\delta K_{\OM_F\cap\polyD}(z,w) \ \leq \ K_F(z,w) \;\; \forall 
(z,w)\in\OM_F\cap(\tfrac{1}{2}\polyD).
\end{equation}
\end{enumerate}
\end{proposition}

Yet, does the condition $(*)$ confer any degree of control on the decay of $F$ near
$z=0$ that is sufficient for optimal estimates; estimates on $K_F(z,w)$ {\em from below} in
particular~? Even if $F$ is radial, $K_F(z,w)\gtrsim\|(z,w)\|^{-2}$ is the best that one
expects (for non-tangential approach) without any {\em additional} information on $F$. To
illustrate: the condition that $(0,0)\in\bdy\OM_F$ is of finite type facilitates optimal
estimates because, with this extra information:
\begin{itemize}
\item[\textbullet] one can find constants $C,\delta>0$, and a $M\in\zahl_+$ such that
\begin{align}
(**) \;\; \ball{2}(0;\delta)\cap\{(z,w):\mi{w}>C|z|^{2M}\}
\subset \OM_F&\cap \ball{2}(0;\delta) \notag \\
\subset \ball{2}(0;\delta)&\cap\{(z,w):\mi{w}>(1/C)|z|^{2M}\}; \notag
\end{align}
\item[\textbullet] one can now make precise estimates by exploiting the simplicity of the prototypal
defining function $z\longmapsto |z|^{2M}$.
\end{itemize}
Some condition that enables one to work --- in the spirit of $(**)$ --- with easier-to-handle 
prototypes of $F$ is called for if one wants optimal estimates in the
infinite-type case. It turns out that we do have useful information if the infinite-type
$F$ satisfies the condition \eqref{E:flatness} spelt out in Theorem \ref{T:main} below.
While this condition might look rather arbitrary, it is in fact a mild restriction. 
It is, in some sense, a signature of $F$ being of infinite type: a domain $\OM_F$ satisfying
\eqref{E:flatness} is necessarily of infinite type at $(0,0)$. The condition
\eqref{E:flatness} encompasses a large class of domains, ranging from the 
``mildly infinite type'' to the very flat at $(0,0)$ (refer to the observations
following Theorem~\ref{T:main}).
\smallskip

We need one further piece of notation. Let 
$f:[0,\infty)\lrarw\RR$ be a strictly increasing function, and let $f(0)=0$. We
define the function $\Lambda_f$ as
\[
\Lambda_f(x) \ := \ \begin{cases}
			-1/\log(f(x)), &\text{if $0<x<f^{-1}(1)$}, \\
			0, &\text{if $x=0$}.
			\end{cases}
\]
We can now state our main theorem. 

\begin{theorem}\label{T:main}
Let $F$ be a $\smoo^\infty$-smooth subharmonic function that vanishes
to infinite order at $0\in\cplx$ and satisfies the condition $(*)$. Suppose the
boundary of the domain $\OM_F:=\{(z,w)\in\CC:\mi w > F(z)\}$ is not Levi-flat around the
origin.
\begin{enumerate}
\item[1)] Define $f$ by the relation $f(|z|)=F(z)$. Then, $f$ is a strictly
increasing function on $[0,\infty)$.
\item[2)] Assume that $F$ satisfies the following condition:
\begin{align}
\text{$\exists$ constants} \ &B,\eps_0>0, 
\text{and a function $\chi\in\smoo([0,\eps_0];\RR)$ s.t.} \notag \\
\; & \chi^p \ \text{is convex on $(0,\eps_0)$ for some $p>0$, and}\qquad\quad \label{E:flatness} \\
(1/B)\chi(x) \ \leq \ &\lamf(x) \ \leq \ B\chi(x) \;\; \forall x\in[0,\eps_0]. \notag
\end{align}
Then, for each $\alpha>0$ and $N\in\zahl_+$, there exists a constant $H_{N,\alpha}>0$,
which depends only on $\alpha$ and $N$; and $C_0, C_1>0$, which 
are independent of all parameters, such that:
\begin{align}\label{E:approach}
C_0(\mi w)^{-2}\left[f^{-1}(\mi w)\right]^{-2} \leq K_F&(z,w)  
\leq \ C_1(\mi w)^{-2}\left[f^{-1}(\mi w)\right]^{-2} \\
&\forall (z,w)\in\apprh, \ 0<\mi w<H_{N,\alpha}, \notag
\end{align}
where $\apprh$ denotes the approach region
\[
\apprh \ := \ \left\{(z,w)\in\OM_F:\sqrt{|z|^2+|\er w|^2}<\alpha(\mi w)^{1/N}\right\}.
\]
\item[3)] Under the assumptions of (2), there exists a constant $H_0>0$ that is
independent of all parameters such that the left-hand inequality in \eqref{E:approach}
in fact holds for all $(z,w)\in \OM_F\cap\{(z,w):\mi w<H_0\}$.
\end{enumerate} 
\end{theorem}
\smallskip

The reader might like to see examples of domains that satisfy all the hypotheses of
Theorem~\ref{T:main}. We discuss two examples, beginning with a very familiar 
example.
\smallskip

\begin{example}\label{Ex:exp} {\em Estimates for the pseudoconvex domain
\[
\OM^\beta \ := \ \{(z,w)\in\CC:\mi w > F_\beta(z)\}
\]
where:
\begin{itemize}
\item[\textbullet] $F_\beta$ is subharmonic;
\item[\textbullet] $F_\beta(z)=\exp(-1/|z|^\beta), \ \beta>0$, in a neighbourhood of $z=0$; and
\item[\textbullet] $F_\beta(z)$ grows like $|z|^2$ for $|z|\gg 1$.
\end{itemize}}

\noindent{We just have to check whether $F$ satisfies the condition \eqref{E:flatness}.
There exists an $\eps_0>0$ such that $\lamf(x)=x^\beta \ \forall x\in[0,\eps_0]$.
We pick
\[
p \ = \ \begin{cases}
	\text{any number $q$ such that $q\beta>1$}, &\text{if $0<\beta\leq 1$,} \\
	1, &\text{if $\beta>1$.}
	\end{cases}
\]
With such a choice for $p$, $(\lamf)^p$ itself is convex on $(0,\eps_0)$.
Hence, Theorem~\ref{T:main} 
tells us that for each $\alpha>0$ and $N\in\zahl_+$, there exists a constant 
$H_{N,\alpha}>0$; and $C_0,C_1>0$, which are independent of all parameters,
such that:
\begin{align}
C_0t^{-2}\left(\log(1/t)\right)^{2/\beta} \leq K_{\OM^\beta}(z,s+it)
\leq &C_1 t^{-2}\left(\log(1/t)\right)^{2/\beta} \notag \\
&\forall (z,s+it)\in\apprh \ \text{and} \ 0< t <H_{N,\alpha}. \qed\notag
\end{align}}
\end{example}
\smallskip

\begin{remark}\label{rem:Lam_f=inf}
We would like to emphasize here that {\em $\lamf$ is allowed to vanish to infinite order at the origin,}
provided it satisfies condition~\eqref{E:flatness}. So, for example, Theorem~\ref{T:main} will provide
optimal growth estimates for $K_F$ for a domain $\OM_F$ of the form \eqref{E:mod} where
\begin{itemize}
\item[\textbullet] $F(z)=\exp\left\{-e^{1/|z|}\right\}$ if $z:0\leq |z|\leq 1/4$; and
\item[\textbullet] $F(z)$ is so defined for $|z|\geq 1/4$ that $F$ satisfies condition $(*)$ and 
$\OM_F$ is pseudoconvex with non-Levi-flat boundary.
\end{itemize} 
Domains like these are what we informally termed  above as ``very flat at $(0,0)$''. The methods
used by Kim and Lee in \cite{kimLee:abBkaicitpd02} do not seem to work for domains like these
precisely because $\lamf$ vanishes to infinite order.
\end{remark}
\smallskip

A few technical preliminaries are needed before a proof of Theorem~\ref{T:main} can 
be given. It would be helpful to get a sense of the key ideas of our proof. A discussion of our
methodology, plus two lemmas, are presented in Section~\ref{S:BergPrelim}. The proof
itself is given in Section~\ref{S:proofMain}. In some sense, our key technical 
preliminary --- without which sharp lower bounds would be tricky to derive --- is
the proof of Proposition~\ref{P:locali}. This proof will form our next section.
\medskip

\section{The proof of Proposition~\ref{P:locali}}

Let $\eta>0$ be as given in the condition $(*)$. We define
\[
\poet \ := \ \begin{cases}
		\text{the least positive integer $\kappa$ such that $\kappa>1/\eta$}, 
		&\text{if $0<\eta\leq 1$,} \\
		1, &\text{if $\eta > 1$}.
		\end{cases}
\]
Define the objects
\begin{align}
\Psi=(\psi_1,\psi_2) \ &: \ (z,w)\ \longmapsto \ \left(\frac{(2i)^\poet z}{(i+w)^\poet},
							\frac{i-w}{i+w}\right), \notag \\
\Pi \ &:= \ \cplx\times\{w\in\cplx:\mi w > -1\}. \notag
\end{align}
Note that $\Psi\in\hol(\Pi;\CC)$ and that $\Psi$ is injective on $\Pi$. 
Define $f$ by the relation $f(|z|)=F(z)$. Then, under our hypotheses, $f$ is
strictly increasing, whence $F(z)\geq 0 \ \forall z\in\cplx$.
The reader is directed to Lemma~\ref{L:inc} for a proof of this
fact. Thus $\OM_F\varsubsetneq\Pi$, whence $\Psi$ is injective on $\OM_F$.
\smallskip

We now claim that $\Psi(\OM_F)$ is bounded. To see this, note that any
$(z,w)\in\OM_F$ can be written as $(z,\er{w}+i(F(z)+h)$, where $h>0$. Thus
\begin{equation}\label{E:psi2bd}
|\psi_2(z,w)|^2 \ = \ \frac{(F(z)+h-1)^2+(\er{w})^2}{(F(z)+h+1)^2+(\er{w})^2} \ \leq
\ 1 \;\; \forall(z,w)\in\OM_F.
\end{equation}
We used the fact that $F(z)\geq 0 \ \forall z\in\cplx$ to deduce this estimate.
Now note that
\[
|\psi_1(z,w)|^2 \ = \ \frac{4^\poet|z|^2}{((F(z)+h+1)^2+(\er{w})^2)^\poet}.
\]
Let $R>0$ and $C>0$ be exactly as given in the condition $(*)$. Then
\[
|\psi_1(z,w)|^2 \ \leq \ \frac{4^\poet R^2}{(F(z)+h+1)^{2\poet}} \ \leq \ 4^\poet R^2 \;\;
\forall(z,w)\in\OM_F \ \text{and $|z|\leq R$.}
\]
On the other hand
\[
|\psi_1(z,w)|^2 \ \leq \ \frac{4^\poet|z|^2}{(F(z))^{2\poet}} \ \leq \ 
\frac{4^\poet|z|^2}{C^{2\poet}|z|^{2\eta\poet}} \;\; \forall(z,w)\in\OM_F \ \text{and $|z|\geq R$.}
\]
From the last two inequalities, we conclude that
\begin{equation}\label{E:psi1bd}
|\psi_1(z,w)|^2 \ \leq \ \max\left\{4^\poet R^2, \ 
\left(\frac{4}{C^2}\right)^\poet R^{-2(\eta\poet-1)}\right\} \;\; \forall(z,w)\in\OM_F.
\end{equation}
From \eqref{E:psi2bd} and \eqref{E:psi1bd}, our claim, and hence Part~(1), follows.
\smallskip

To demonstrate Part~(2), we will need a localization principle established by Ohsawa:
\begin{itemize}
\item[{}] {\bf {\em Localization Lemma} (Ohsawa, \cite{ohsawa:bbBkfpd84})} 
{\em Let $D$ be a bounded pseudoconvex domain
in $\Cn$, $p$ be a boundary point, and $V\Subset U$ be two open neighbourhoods of $p$.
Then, there is a constant $\delta\equiv\delta(U,V)>0$ such that
\[
\delta K_{D\cap U}(Z) \ \leq \ K_D(Z) \;\; \forall Z\in D\cap V.
\]}
\end{itemize}

\noindent{Substituting
\begin{align}
D \ &= \ \Psi(\OM_F), && \qquad p \;\; = \ (0,1), \notag \\
D\cap U \ &= \ \Psi(\OM_F\cap\polyD), &&D\cap V \ = \ 
\Psi\left(\OM_F\cap\left(\tfrac{1}{2}\polyD\right)\right) \notag
\end{align}
into the localization lemma, we conclude that there exists a $\delta\equiv\delta(\polyD)>0$
such that (here $G_F$ stands for $\Psi(\OM_F)$)
\begin{equation}\label{E:locali1}
\delta K_{G_F\cap U}(\Psi(z,w)) \ \leq \ K_{G_F}(\Psi(z,w)) \;\; \forall
(z,w)\in\OM_F\cap(\tfrac{1}{2}\polyD).
\end{equation}
Recall, however, the transformation rule for the Bergman kernel:
\[
K_{\OM^j}(z,w) \ = \ |{\rm Jac}_\cplx(\Psi)(z,w)|^2 K_{\Psi(\OM^j)}(\Psi(z,w)) \;\;
\forall (z,w)\in\OM^j, \ j=1,2,
\]
where, in the present case, $\OM^1=\OM_F$ and $\OM^2=\OM_F\cap\polyD$. Applying this
to \eqref{E:locali1} gives us the inequality \eqref{E:locali}.\qed
\medskip

\section{Preliminary remarks and lemmas}\label{S:BergPrelim}

The idea behind the upper bound in \eqref{E:approach} is quite standard. Given a point $(z_0,w_0)\in\OM_F$,
the quantity $K_F(z_0,w_0)$ is dominated by the reciprocal of the volume of the largest polydisc centered
at $(z_0,w_0)$ that is contained in $\OM_F$. The volume of this polydisc will be influenced by the curvature
of $\bdy\OM_F$ at the point on $\bdy\OM_F$ that is closest to $(z_0,w_0)$. However, if $(z_0,w_0)$ is confined
to any of the approach regions $\apprh$, then this volume is controlled by the boundary geometry
at $(0,0)\in\bdy\OM_F$. That one has this control for any approach region $\apprh$ --- regardless
of $\alpha$ and $N$ --- is a consequence of the fact that $(0,0)$ is of infinite type.
\smallskip

The derivation of the lower bound in \eqref{E:approach} relies on the construction of a 
suitable square-integrable holomorphic function. In this construction, we are aided by the
localization principle stated in Proposition~\ref{P:locali}. 
The three main ingredients in the derivation of the lower bound are:
\begin{itemize}
\item[{\em i)}] We choose a suitable polydisc $\polyD$ centered at $(0,0)$ and estimate
$K_{\OM_F\cap\polyD}(z,w)$ for $(z,w)\in\apprh\cap\left(\tfrac{1}{2}\polyD\right)$. We rely
on the fact that $K_{\OM_F\cap\polyD}(z,w)$ is given by
\[
K_{\OM_F\cap\polyD}(z,w) \ = \ \sup\left\{\frac{|\phi(z,w)|^2}{\|\phi\|^2_{\elII({\loc})}}:
					\phi\in A^2(\OM_F\cap\polyD)\right\}.
\]
\item[{\em ii)}] To obtain a lower bound, we select a suitable function $\phi_t\in A^2(\OM_F\cap\polyD)$
and estimate $\|\phi_t\|^2_{A^2}, \ t>0$. This reduces
finding a lower bound for $K_{\OM_F\cap\polyD}(z,s+it)$, 
$(z,s+it)\in\apprh\cap\left(\tfrac{1}{2}\polyD\right)$, to estimating 
an integral over a region in $\RR^4$ whose boundaries are determined by the function $f$.
\item[{\em iii)}] The difficult issue is to find the {\em desired bound in terms of $t$} for 
the latter integral. The condition~\eqref{E:flatness} is used to break up the aforementioned 
region of integration into sub-domains on which the integral admits the desired estimate.
\end{itemize}
\smallskip

We now present two lemmas that will be necessary to complete the proof of 
Theorem~\ref{T:main}. Lemma~\ref{L:inc} constitutes the proof of Part~(1) of
Theorem~\ref{T:main}. 
\smallskip

\begin{lemma}\label{L:inc}
Let $F$ be a smooth subharmonic function on $\cplx$ such that $F(0)=0$ and $F$ is
radial. Define $f$ by the relation
$f(|z|)=F(z)$, and write $\OM_F:=\{(z,w)\in\CC:\mi w > F(z)\}$.
Assume that $\OM_F$ is not Levi-flat in a neighbourhood of $(0,0)$. Then $f$ is a strictly
increasing function on $[0,\infty)$.
\end{lemma}
\begin{proof}
Suppose there exist $r_1<r_2$, with $r_1,r_2\in[0,\infty)$, such that $f(r_1)\geq f(r_2)$.
Then, by our hypothesis on $F$
\[
\sup_{z\in\bdy D(0;r_2)}F(z) \ = \ f(r_2) \ \leq \ f(r_1).
\]
By the Maximum Principle, therefore, $F|_{D(0,r_2)}\equiv 0$. But then, this would imply
that the portion $\bdy\OM_F\cap\ball{2}(0;r_2)$ of $\bdy\OM_F$ is Levi-flat; i.e. a
contradiction. Hence $f$ is strictly increasing.
\end{proof}
\smallskip

\begin{lemma}\label{L:bigKey}
Let $F$ and $\OM_F$ have all the properties listed in Lemma~\ref{L:inc}. Let $\lamf$ satisfy
the condition~\eqref{E:flatness}, and let (since, in view of Lemma~\ref{L:inc}, $f$
is increasing) $\geef:=\lamf^{-1}$. Then, there exist constants $T>0$ and $K>0$ such that 
\begin{equation}\label{E:compare}
0 \ < \ \geef(2t)^2-\geef(t)^2 \ \leq \ K\geef(t)^2 \;\; \forall t\in(0,T).
\end{equation}
\end{lemma}
\begin{proof}
We just have to show that there exist $T>0$ and $M>0$ such that
\begin{equation}\label{E:prelim}
0 \ < \ \geef(2t)-\geef(t) \ \leq \ M\geef(t) \;\; \forall t\in(0,T).
\end{equation}
If we could show this, then it would follow that
\[
\geef(2t)^2-\geef(t)^2 \ \leq \ M(M+2)\geef(t)^2 \quad \forall t\in(0,T). \notag
\]
To proceed further, we need the following:
\smallskip

\noindent{{\bf Fact.} {\em Let $g$ be a continuous, strictly increasing function on $[0,R]$
satisfying $g(0)=0$, and assume $g^p$ is convex on $(0,R)$ for some $p>0$. Define 
$G:=g^{-1}$, and let $B>1$. Then:}
\begin{equation}\label{E:ratio}
\frac{G(Bt)}{G(t)} \ \leq \ B^p \;\; \forall t\in (0,g(R)/B).
\end{equation}} 

\noindent{To verify this fact, set $\Phi:=(g^p)^{-1}$. Then:
\begin{equation}\label{E:simplify}
\Phi(t) \ = \ G(t^{1/p}) \quad \forall t\in[0,g(R)^p],
\end{equation}
By hypothesis, $\Phi$ is concave on $(0,g(R)^p)$. But since $\Phi$ is also continuous,
\[
\frac{\Phi(B^{p}t^p)}{B^{p}t^p} \ \leq \ \frac{\Phi(t^p)}{t^p} \;\; \forall
t\in (0,g(R)/B).
\]
The above fact now follows simply by rearranging the terms in the above inequality, and
applying \eqref{E:simplify}.}
\smallskip
 
Now let $\eps_0$, $B$, $\chi$, and $p$ be as in \eqref{E:flatness}. Let us also define
\begin{align}
\varkappa_0 &:= (\chi)^{-1} : [0,\chi(\eps_0)]\lrarw \RR \notag \\
\varkappa_1 &:= (B\chi)^{-1} : [0,B\chi(\eps_0)]\lrarw \RR \notag \\
\varkappa_2 &:= \ \left[(1/B)\chi\right]^{-1} : [0,(1/B)\chi(\eps_0)]\lrarw \RR \notag
\end{align}
Since $\lamf$ and $\chi$ are strictly increasing (in view of Lemma~\ref{L:inc}), 
condition~\eqref{E:flatness} implies
that:
\[
\varkappa_1(t) \ \leq \ \geef(t) \ \leq \ \varkappa_2(t) \;\; \forall t\in[0,T_1],
\]
where $T_1:=(1/B)\chi(\eps_0)$. Therefore, we get
\begin{equation}\label{E:sandwich}
\frac{\geef(2t)}{\geef(t)} \ \leq \ 
\frac{\varkappa_2(2t)}{\varkappa_1(t)} \ = \ 
\frac{\varkappa_2(2t)}{\varkappa_2(t)} \ \frac{\varkappa_2(t)}{\varkappa_1(t)} \;\;
\forall t\in(0,T_1/2).
\end{equation}
We now note that
\[
\varkappa_1(t) \ = \ \varkappa_0(B^{-1}t), \quad 
\varkappa_2(t) \ = \ \varkappa_0(Bt) \quad \forall t\in[0,T_1].
\]
Given this last piece of information, we can apply the inequality \eqref{E:ratio} to
the ratios on the right-hand side of \eqref{E:sandwich}. Set $T:=\min(T_1/2,(1/B)T_1)$.
Then,
\[
\frac{\geef(2t)}{\geef(t)} \ \leq \ (2B^2)^p
\;\; \forall t\in(0,T).  
\]
From this, the estimate \eqref{E:prelim} clearly follows if we take
$M=(2B^2)^p-1$. Hence, by our earlier remarks, the result follows.
\end{proof}
\medskip

\section{The proof of Theorem~\ref{T:main}}\label{S:proofMain}

Part~(1) of Theorem~\ref{T:main} has already been established in Lemma~\ref{L:inc}.
Therefore, $f^{-1}$ is a well-defined function. Observe that
\begin{equation}\label{E:inverse}
f^{-1}(t) \ = \ \geef\left(\frac{1}{\log(1/t)}\right), \quad 0<t<1.
\end{equation}
Let $R>0$ be so small that
\[
f^{-1}(3t/4) \ = \ \geef\left(\frac{1}{\log(4/3)+\log(1/t)}\right) \ \geq \ 
\geef\left(\frac{1}{2\log(1/t)}\right) \;\; \forall t\in(0,R).
\]
Let $M$ and $T$ be as given by \eqref{E:prelim} above. Shrinking
$R>0$ if necessary so that $0<1/2\log(1/t)<T \ \forall t\in(0,R)$, we get
\begin{equation}\label{E:compare1}
\frac{f^{-1}(3t/4)}{f^{-1}(t)} \ \geq \ \frac{\geef(1/2\log(t^{-1}))}{\geef(1/\log(t^{-1}))} \ 
\geq \ (M+1)^{-1} \;\;\forall t\in[0,R).
\end{equation}
Write $\mu:=(M+1)^{-1}$. We are given $\alpha>0$ and $N\in\zahl_+$. Since $f(x)$ vanishes to infinite order
at $x=0$, there exists a $H_{N,\alpha}>0$ such that $R\geq H_{N,\alpha}$ and
\begin{equation}\label{E:compare2}
\alpha t^{1/N} \ \leq \ \frac{\mu}{4}f^{-1}(t) \quad\forall t\in[0,H_{N,\alpha}).
\end{equation}
From \eqref{E:compare1} and \eqref{E:compare2}, we see that
\[
|z|+\frac{\mu}{2}f^{-1}(t) \ < \ f^{-1}(3t/4)\;\;\forall z:0\leq|z|<\alpha t^{1/N}, \ 
0<t<H_{N,\alpha},
\]
whence the polydisc
\begin{multline}
\polyD(z,t) \ := \ \Dee\left(z;\frac{\mu}{2}f^{-1}(t)\right)\times\Dee(it;t/4) \ \subset \ \OM_F \\
\forall z:0\leq|z|<\alpha t^{1/N}, \ 0<t<H_{N,\alpha}.
\end{multline}
Now note that that translations $T_s:(z,w)\longmapsto (z,s+w), \ s\in\RR$, are all
automorphisms of $\OM_F$. Thus, by the transformation rule for the Bergman
kernel, and by monotonicity, we get

\begin{align}
K_F(z,s+it) \ &= \ K_F(z,it) \notag \\
		&\leq \ K_{\polyD(z,t)}(z,it) \notag \\ 
		&= \ \frac{1}{{\rm vol}\left(\polyD(z,t)\right)} \;\; 
		\forall (z,s+it)\in\apprh, \ 0<t<H_{N,\alpha}. \notag
\end{align}
The last equality follows from the fact that $\polyD(z,t)$ is a Reinhardt domain
centered at $(z,t)$. Hence, we have one half of the estimate~\eqref{E:approach}:
\begin{align}
K_F(z,w)
\leq \ C_1(\mi w)^{-2}\left[f^{-1}(\mi w)\right]^{-2}\;\;
\forall &(z,w)\in\apprh, \label{E:1stHalf} \\
&0<\mi w<H_{N,\alpha}, \notag
\end{align}
where $C_1=64/\mu^2\pi^2$. 
\smallskip

We will now derive a lower bound. We set $A:=\min(f^{-1}(1),1)$. For the remainder of
this proof, $\polyD$ will denote the polydisc $\Dee(0;A)\times\Dee(0;1)$. In view of
the inequality \eqref{E:locali} of Proposition~\ref{P:locali}, it suffices to find
a lower bound for $K_{\OM_F\cap\polyD}(z,w)$ for $(z,w)\in\left(\tfrac{1}{2}\polyD\right)$.
It is well known that
\begin{equation}\label{E:maxim}
K_{\OM_F\cap\polyD}(z,w) \ = \ \sup_{\phi\in A^2(\OM_F\cap\polyD)}
				\frac{|\phi(z,w)|^2}{\|\phi\|^2_{\elII({\loc})}}.
\end{equation}
Once again, we use the fact that the translations $T_u:(z,w)\longmapsto (z,u+w), \ u\in\RR$, 
are all automorphisms of $\OM_F$, whence
\begin{equation}\label{E:translate}
K_F(z,s+it) \ = \ K_F(z,(u+s)+it) \quad\forall(z,s+it)\in\OM_F \ \text{and $\forall u\in\RR$}.
\end{equation}
Set $\phi_t(z,w):=-4t^2/(w+it)^2, \ t>0.$ Then, from the localization principle \eqref{E:locali}, and 
from \eqref{E:maxim} and \eqref{E:translate}, we get
\begin{align}\label{E:lbdK1}
K_F(z,s+it) \ &\geq \ \delta K_{\OM_F\cap\polyD}(z,it) \\
		&\geq \ \frac{\delta}{\|\phi_t\|^2_{\elII({\loc})}}\;\; 
		\forall (z,t)\in\left(\tfrac{1}{2}\polyD\right).\notag
\end{align}
Let us write $w=u+iv$. We leave the reader to verify that we can apply Fubini's theorem
wherever necessary in the following computation:
\begin{align}
\|\phi_t\|^2_{\elII({\loc})} \ &= \
\lint{|z|<A}{{}} \ \lint{-1}{1}\lint{F(z)}{\sqrt{1-u^2}}\frac{16t^4}{|u+i(v+t)|^4}dv \ du \ dA(z)
\notag \\
&\leq \
16t^4\lint{|z|<A}{{}} \ \lint{F(z)}{\infty} \ \lint{-1}{1}(v+t)^{-4}
\left(1+\left(\frac{u}{v+t}\right)^2\right)^{-2}du \ dv \ dA(z) \notag \\
&\leq \
8t^4\left(\lint{\mathbb{R}}{{}}\frac{dX}{(1+X^2)^2}\right)
\lint{|z|<A}{{}}(t+F(z))^{-2}dA(z)\notag \\
&= \ Ct^4 \lint{0}{A}\frac{r}{(t+f(r))^2}dr, \notag
\end{align}
where $C>0$ is a universal constant. In what follows, we shall denote $f^{-1}(s)$ by
$R_{s}$. By equation \eqref{E:inverse} we have
\begin{equation}\label{E:inverse2}
R_{\sqrt{t}} \ = \ \geef\left(\frac{2}{\log(1/t)}\right), \quad 0<t<1.
\end{equation}
We break up the interval of integration
of the last integral into three sub-intervals to compute:
\begin{align}
\|\phi_t\|^2_{\elII({\loc})} \ &= \
Ct^4 \left(\int_0^{R_t} + \ \int_{R_t}^{R_{\sqrt{t}}} + \ \int_{R_{\sqrt{t}}}^A\frac{r}{(t+f(r))^2}dr\right)
\notag \\
&\leq \ Ct^4\int_0^{R_t}\frac{r}{t^2}dr + 
Ct^4 \left(\int_{R_t}^{R_{\sqrt{t}}} + \ \int_{R_{\sqrt{t}}}^A
\frac{r}{4tf(r)}dr\right) \notag \\
&\leq \ \frac{C}{2}t^2(R_t)^2 + \frac{C}{4}t^2\int_{R_t}^{R_{\sqrt{t}}}r \ dr + 
\frac{C}{4}t^{5/2}A(A-R_{\sqrt{t}}) \notag \\
&\leq \ \frac{C}{2}t^2(R_t)^2 + \frac{C}{4}t^{5/2}A(A-R_{\sqrt{t}}) \label{E:integ1} \\
&\qquad\quad+\frac{C}{8}t^{2}\left(\geef^2\left(\frac{2}{\log(1/t)}\right)-
\geef^2\left(\frac{1}{\log(1/t)}\right)\right),\;\; 0<t<1. \notag
\end{align} 
We used the relation \eqref{E:inverse2} in the estimate for the middle integral above.   
\smallskip 

We now apply Lemma~\ref{L:bigKey} to the third term in \eqref{E:integ1}. Let 
$T>0$ and $K>0$ be defined as given by Lemma~\ref{L:bigKey}. Let $H_0$ be so small that 
$1/\log(t^{-1})<T \ \forall t\in(0,H_0)$, {\em and} so that the second inequality below
holds true:
\begin{align}\label{E:integ2}
\|\phi_t\|^2_{\elII({\loc})} \ &\leq \
\frac{C}{2}\left(1+\frac{K}{4}\right)t^2(R_t)^2 + \frac{C}{4}t^{5/2} \\
&\leq \ C(1+K/4)t^2(R_t)^2 \;\; \forall t\in(0,H_0).\notag
\end{align}
Since $f(x)$ vanishes to infinite order at $x=0$, we can lower $H_0$ --- and this is independent
of parameters like $\alpha>0$ and $N\in\zahl_+$ --- so that the first term of the first inequality
above dominates the second for $t\in(0,H_0)$, giving us \eqref{E:integ2}. Lowering the value of 
$H_0$ further if necessary, we also ensure: 
\[
\OM_F\cap\{(z,w):\mi w<H_0\} \ \subset \ \OM_F\cap\polyD.
\]
From \eqref{E:lbdK1} and \eqref{E:integ2}, we conclude that there exists a constant
$C_0$, which is independent of all parameters, such that
\[
C_0(\mi w)^{-2}\left[f^{-1}(\mi w)\right]^{-2} \leq K_F(z,w) \;\;
\forall (z,w)\in \OM_F\cap\{(z,w):\mi w<H_0\}.
\]
This establishes Part~(3) of our theorem. As a special case, we get the lower bound
on $K_F(z,w)$ in the estimate \eqref{E:approach}. Along with \eqref{E:1stHalf}, this
establishes Part~(2) of our theorem. \qed
\bigskip

\noindent{{\bf Acknowledgements.} First, and foremost, I thank Alexander Nagel for his
interest in an earlier version of this article. My discussion with him led to the
simplification of several arguments in that version; now the arXiv preprint 
\texttt{arXiv:0708.2894v1}. I also thank all my colleagues who noticed this preprint and
offered their comments during my stay at the Institut Mittag-Leffler during the Special
Semester on complex analysis in several variables. The hospitality of the Institut
Mittag-Leffler is greatly appreciated. Finally, I thank the anonymous referee of this
article for his/her valuable suggestions.}

\end{document}